\UseRawInputEncoding
\documentclass[10pt,reqno]{amsart}
\usepackage{amsmath, amssymb, amsfonts, amsthm, xcolor}
\usepackage{mathtools}
\usepackage{hyperref}
\newtheorem{thm}{Theorem}[section]

\newtheorem{prop}[thm]{Proposition}
\theoremstyle{definition}
\newtheorem{dfn}[thm]{Definition}
\newtheorem{exple}[thm]{Example}

\newtheorem{remark}[thm]{Remark}
\theoremstyle{plain}
\newtheorem{cor}[thm]{Corollary}

\numberwithin{equation}{section}
\usepackage[a4paper, top = 1.2in, bottom = 1.2in, left = 1.2in, right = 1.2in]{geometry}
\numberwithin{equation}{section}

\newcommand{\N}{\mathbb{N}}
\newcommand{\Q}{\mathbb{Q}}

\newcommand{\Z}{\mathbb{Z}}

\newcommand{\F}{\mathbb{F}}

\newcommand{\m}{\mathfrak{m}}
\newcommand{\n}{\mathfrak{n}}
\newcommand{\p}{\mathfrak{p}}

\newcommand{\mfn}{\mathfrak{n}}

\newcommand{\SL}{\mathrm{SL}}
\newcommand{\GL}{\mathrm{GL}}

\newcommand{\mrm}[1]{\mathrm{#1}}

\def\1{1\!\!1}

\newcommand{\psmat}[4]{\bigl( \begin{smallmatrix} #1 & #2 \\ #3 & #4 \end{smallmatrix} \bigr)}

\title[On Congruences and Linear relations]{On Congruences and Linear relations 
for Drinfeld modular forms of level $\Gamma_0(T)$, arbitrary type}
\author[T. Dalal]{Tarun Dalal}
\email{ma17resch11005@iith.ac.in}
\address{
Department of Mathematics \\
Indian Institute of Technology Hyderabad\\
Kandi, Sangareddy - 502285\\
INDIA. 
}
\author[N. Kumar]{Narasimha Kumar}
\email{narasimha@math.iith.ac.in}
\address{
Department of Mathematics \\
Indian Institute of Technology Hyderabad\\
Kandi, Sangareddy - 502285\\
INDIA. 
}

\date{\today}
\begin{document}
\begin{abstract}
In~\cite{Cho09}, Choi studied congruences of coefficients 
(modulo $T^q-T$) for Drinfeld modular forms of level $\Gamma_0(T)$,  trivial type
and  the linear relations between the initial coefficients of those.
In this article, we generalize these results for level $\Gamma_0(T)$, arbitrary type.
\end{abstract}

\maketitle

\section{Introduction}
\label{Introduction}
 The study of congruences of Fourier coefficients of modular forms 
 is an  interesting area of research in number theory. There were several important works in the literature on the congruences
 of Fourier coefficients (cf.~\cite{Rib84},~\cite{Hid85}), but we are particularly focusing on~\cite{CKO05}, where the authors studied
 the $p$-divisibility properties of Fourier coefficients of modular forms for $\SL_2(\Z)$.  As an application, they have retrieved the results of Hatada, Hida on non-ordinary primes. 
 Later, these results were generalized by El-Guindy 
 for cusp forms of level $2,3,5,7 $ and $13$ (cf.~\cite{Gui07}). 
 
The study of sign changes of Fourier coefficients of modular forms is also an active area of  research in modern number theory (cf.~\cite{Mur83},~\cite{KS06}). In~\cite{Sie69}, Siegel obtained an upper bound for the first sign change of Fourier coefficients of modular forms for $\SL_2(\Z)$ by studying the space of linear relations among the initial Fourier coefficients of modular forms for $\SL_2(\Z)$. A general result on the linear relations were
obtained by Choie, Kohnen, and Ono for $\SL_2(\Z)$ (cf.~\cite{CKO05}), by El-Guindy for cusp forms of level $2,3,5,7 $ and $13$ (cf.~\cite{Gui07}).

 
 It is of real interest  to answer  similar questions  for the coefficients in the $u$-series expansion 
 of Drinfeld modular forms. Such a study has been initiated
 by Choi. In~\cite{Cho08}, she studied ($T^q-T$)-divisibility properties of the coefficients  of Drinfeld modular forms of any weight, any type for $\GL_2(A)$, and determined all the linear relations between the initial coefficients  of Drinfeld modular forms of trivial type. 
 In~\cite{Kaz08}, Kazalicki also obtained similar results for $\GL_2(A)$.

 For the level $\Gamma_0(T)$, these results are only known for Drinfeld modular forms of any weight, trivial type.
 In~\cite{Cho09}, Choi studied $(T^q-T)$-divisibility properties of the coefficient of  Drinfeld modular forms 
 of any weight, trivial type for $\Gamma_0(T)$. In \textit{loc. cit.}, she determined all the linear relations between the initial coefficients in the $u$-series expansion of Drinfeld modular forms of trivial type, level $\Gamma_0(T)$.

In this article, we continue this study for $\Gamma_0(T)$ by generalizing the results of~\cite{Cho09} for 
Drinfeld modular forms of any weight, \textbf{arbitrary type},  level $\Gamma_0(T)$.
Throughout the article, we fix to use the following notations:
\begin{itemize}
\item $p$ is an odd prime number and $q=p^r$ for some $r \in \N$.
\item $k\in \N$ and $l\in \Z/(q-1)\Z$ such that $k\equiv 2l \pmod {q-1}$.
      Let $0 \leq l \leq q-2$ be a lift of $l \in \Z/(q-1)\Z$. By abuse of notation,
      we continue to write $l$ for the integer as well as its class. Then, we define 
      $r_{k, l} := \frac{k-2l}{q-1}$  and $r_{k, l,N} := r_{k,l} + N + 1$, where $N$ is a non-negative integer. Recall that $\dim M_{k,l}(\Gamma_0(T))=1+r_{k,l}$ (cf. \cite[Proposition 4.1]{DK}).
\end{itemize}
\subsection{An overview of the article }
The article is organized as follows. In \S \ref{Basic Theory} we recall the basic theory of Drinfeld modular forms. In \S \ref{Main_Section_1} we study congruences of coefficients and generalize \cite[Theorem 3.4]{Cho09} for Drinfeld modular forms of any weight, \textbf{arbitrary type},  level $\Gamma_0(T)$. 
Finally, in \S \ref{Sec_Linear Relations} we determine all the  linear relations between the initial coefficients in $u$-series expansion of Drinfeld modular forms of any weight,  \textbf{arbitrary type} for  $\Gamma_0(T)$ which generalizes \cite[Theorem 4.1]{Cho09}.

\section{Basic theory of Drinfeld modular forms}
\label{Basic Theory}

In this section, we shall recall some basic theory of Drinfeld modular forms
(see \cite{Gos80},\cite{Gos80a}, \\ \cite{Gek88},\cite{GR96} for more details).

Let $\F_q$ denote the finite field of order $q$. Then, we set $A=\F_q[T]$ 
and
$K=\F_q(T)$.
Let $K_\infty=\F_q((\frac{1}{T}))$ be the completion of $K$ 
with respect to the infinite place $\infty$ (corresponding to $\frac{1}{T}$-adic valuation) and denote by $C$ the completion of an algebraic closure of $K_{\infty}$.
Let $L=\tilde{\pi}A \subseteq C$ be the $A$-lattice of rank $1$,
corresponding to the rank $1$ Drinfeld module given by $\rho_T=TX+X^q,$ where $\tilde{\pi}\in K_\infty(\sqrt[q-1]{-T})$ is defined up to a $(q-1)$-th root of unity.
The Drinfeld upper half-plane $\Omega= C-K_\infty$
has a rigid analytic structure. 
The group $\GL_2(K_\infty)$ acts on $\Omega$ via fractional linear transformations.
  Any $x\in K_\infty^\times$ has the unique expression 
$x= \zeta_x\big(\frac{1}{T} \big)^{v_\infty(x)}u_x,$
where $\zeta_x\in \F_q^\times$, and $v_\infty(u_x-1)> 0$ ($v_\infty$ is the valuation at $\infty$).
For any $k\in \N,l\in \Z/(q-1)\Z, \gamma=\psmat{a}{b}{c}{d}\in \GL_2(K_{\infty})$, and $f:\Omega \longrightarrow C$,  we define 
$f|_{k,l} \gamma := \zeta_{\det\gamma}^{l}\big(\frac{\det \gamma}{\zeta_{\det(\gamma)}} \big)^{k/2}(cz+d)^{-k}f(\gamma z).$
For an ideal $\mfn \subseteq A$, let $\Gamma_0(\mfn)$  denote
the congruence subgroup $\Gamma_0(\mfn) :=\{\psmat{a}{b}{c}{d}\in \mathrm{GL}_2(A) : c\in \mfn \}.$
\begin{dfn}
A rigid holomorphic (resp., meromorphic) function $f:\Omega \longrightarrow C$ is said to be a holomorphic (resp., meromorphic) Drinfeld modular form for $\Gamma_0(\n)$
of weight $k$ and type $l$ if 
\begin{enumerate}
\item $f|_{k,l}\gamma= f$ , $\forall \gamma\in \Gamma_0(\mfn)$,
\item $f$ is holomorphic (resp., meromorphic) at the cusps of $\Gamma_0(\mfn)$.
\end{enumerate}
\end{dfn}
Every holomorphic (resp., meromorphic) Drinfeld modular form has an unique expansion at the cusp $\infty$ w.r.t the parameter $u(z) :=  \frac{1}{e_L(\tilde{\pi}z)},$ where $e_L(z):= z{\prod_{\substack{0 \ne \lambda \in L }}}(1-\frac{z}{\lambda})$. Let $M_{k,l}(\Gamma_0(\mfn))$ denote the $C$-vector space of holomorphic 
Drinfeld modular forms of weight $k$ and type $l$ for $\Gamma_0(\mfn)$. If $f\in M_{k,l}(\Gamma_0(\mfn))$ vanishes at the cusps of $\Gamma_0(\mfn)$, then we say that $f$ is a Drinfeld cusp form.
Any $f\in M_{k,l}(\Gamma_0(\mfn))$ has $u$-series expansion at $\infty$ of the form 
$\sum_{0 \leq \ i \equiv l \mod (q-1)}a_f(i)u^{i}$
(similarly, for meromorphic function $f$, the index $i$ starts from a negative integer).  

By definition, if $k\not \equiv 2l \pmod {q-1}$ then $M_{k,l}(\Gamma_0(\mfn))=\{0\}$. Hence, we always with the vector space $M_{k,l}(\Gamma_0(\n))$ where $k\in \N$ and $l\in \Z/(q-1)\Z$ such that $k\equiv 2l \pmod {q-1}$.
We now give some useful examples of Drinfeld modular forms.
\begin{exple}[\cite{Gos80}, \cite{Gek88}]
\label{Eisenstein Series}
Let $d\in \N$. For $z\in \Omega$, the function
\begin{equation*}
g_d(z) := (-1)^{d+1}\tilde{\pi}^{1-q^d}L_d \sum_{\substack{a,b\in \F_q[T] \\ (a,b)\ne (0,0)}} \frac{1}{(az+b)^{q^d-1}}
\end{equation*}
is a Drinfeld modular form of weight $q^d-1$ and type $0$ for $\mathrm{GL}_2(A)$,
where $\tilde{\pi}$ is the Carlitz period and $L_d:=(T^q-T)\ldots(T^{q^d}-T)$ is the least common multiple of all monics of degree $d$.
We refer $g_d$ as an Eisenstein series and it does not vanish at $\infty$.
\end{exple}

\begin{exple}[Poincar\'e series]
\label{Poincare Series}
In~\cite{Gek88}, Gekeler defined the Poincar\'e series as follows:
$
 h(z) = \sum_{\gamma\in H\char`\\ \GL_2(A)} \frac{\det \gamma . u(\gamma z)}{(cz+d)^{q+1}}, 
$
where $H=\big\{\psmat{*}{*}{0}{1}\in \GL_2(A)\big\}$ and $\gamma = \psmat{a}{b}{c}{d}\in \GL_2(A)$.
Then $h$ is a cusp form of weight $q+1$, type $1$ for $\mathrm{GL}_2(A)$.  
The $u$-series expansion of $h$ at $\infty$ is given by $ -u -u^{(q-1)^2+1}+\ldots$.
By the properties of $\Delta$-function in~\cite[Page 228]{Gek86a}, we deduce  that $h$ vanishes exactly once (resp., $q$-times) at $\infty$ (resp., at $0$) as a Drinfeld modular form of level $\Gamma_0(T)$.

\end{exple}



We end this section by introducing an important function $E$. In~\cite{Gek88}, Gekeler defined the function $E(z):= \frac{1}{\tilde{\pi}} \sum_{\substack{a\in \F_q[T] \\ a \ \mathrm{monic}}} ( \sum_{b\in \F_q[T]} \frac{a}{az+b} )$
 which is analogous to the Eisenstein series of weight $2$ over $\Q$. Though $E$ is not modular but we can use it to construct a Drinfeld modular form
\begin{equation}
\label{E_T}
E_T(z) := E(z)- TE(Tz) \in M_{2,1}(\Gamma_0(T)).
\end{equation} 
The $u$-series expansion of $E_T$ at $\infty$ is given by $u-Tu^q+\ldots$
(cf.~\cite[Proposition 4.3]{DK21} for a detailed discussion about this function).

\section{Congruences for coefficients of Drinfeld Modular forms} 
\label{Main_Section_1}
%
%

In this section, we generalize~\cite[Theorem 3.4]{Cho09} to Drinfeld modular forms $f$ of level $\Gamma_0(T)$, arbitrary type. 
 We start by introducing the modular forms $\Delta_T$ and $\Delta_W$. 
Recall that $0$ and $\infty$ are the only cusps of the Drinfeld modular curve $X_0(T):= \overline{\Gamma_0(T)\char`\\ \Omega}$ and the operator $W_T:=\psmat{0}{-1}{T}{0}$ permutes the cusps.
Consider the functions
$ \Delta_T(z) := \frac{g_1(Tz)-g_1(z)}{T^q-T}\ \mrm{and} \ 
\Delta_W(z) := \frac{T^qg_1(Tz)-Tg_1(z)}{T^q-T}=-T^{\frac{q+1}{2}}\Delta_T|_{q-1,0}W_T$. 
Note that,  $\Delta_T$, $\Delta_W\in M_{q-1,0}(\Gamma_0(T))$ and their $u$-expansions are given by 
$\Delta_T =u^{q-1}-u^{q(q-1)}+\ldots\in A[[u]]$ and
$\Delta_W = 1+ Tu^{q-1} - T^qu^{q(q-1)} + \ldots \in A[[u]]$.

\begin{prop}\cite[Proposition 4.3]{DK}
\label{IMP_Prop}
Let $\Delta_T, \Delta_W$ and $E_T$ be as defined before.
\begin{enumerate}
\item \label{P1}  \label{P2}
       The modular form $\Delta_T$(resp., $\Delta_W$) vanishes $q-1$ times at  $\infty$ (resp., at  $0$) and non-zero on $\Omega \cup \{0\}$ (resp.,  on $\Omega \cup \{\infty\}$). Hence,  
       $\Delta_T$ and $\Delta_W$ are algebraically independent.
 \item \label{P3}
        The set $S:= \{\Delta_W^{r_{k,l}}E_T^l, \Delta_W^{r_{k,l}-1}\Delta_TE_T^l, \ldots, \Delta_W\Delta_T^{r_{k,l}-1} E_T^l, \Delta_T^{r_{k,l}}E_T^l \}$ forms a basis for the $C$-vector space $M_{k,l}(\Gamma_0(T))$.
 \item  \label{P4}
        For any $k \in \N$, the mapping $\eta: M_{k-2l,0}(\Gamma_0(T)) \longrightarrow M_{k,l}(\Gamma_0(T))$ 
         defined by $f \longrightarrow fE_T^l $
         is an isomorphism. 
 \item \label{P5} \label{P6}
        We have $E_T^{q-1}= \Delta_W\Delta_T$.  In particular, the function $E_T\in M_{2,1}(\Gamma_0(T))$ vanishes exactly once at the cusps $0,\infty$ and non-vanishing elsewhere.
\item \label{P7} $h(z)= -\Delta_W(z)E_T(z)$.
\end{enumerate}
\end{prop}

In~\cite[Theorem 3.4]{Cho09}, Choi multiplied $f$ with $E$ to get the congruence for the coefficients of $fE$ modulo $T^{q}-T$, which imply  the congruences for the coefficients of $f$ (cf. Corollary~\cite[Corollary 3.5]{Cho09}). Although the methodology of the proof of the main result is similar, the novelty in our work is to multiply $f$ with $E_T^{q-l}$ to get the congruence for the coefficients of $fE_T^{q-l}$ modulo $T^{q^d}-T$ and deduce the congruences for the coefficients of $f$.

Before we state our main result of this section, let us recall some important results which are useful in the proof. First, we shall recall the following Theorem
(cf.~\cite[Theorem 7.14.2]{Har77}).
 
\begin{thm}[Residue Theorem] 
\label{Residue Theorem}
Let $\n$ be an ideal of $A$.
For any $1$-form $\omega$ on $X_0(\n):= \overline{\Gamma_0(\n)\char`\\ \Omega}$, we have $\sum_{p\in X_0(\n)}\mathrm{Res}_p\omega=0.$
\end{thm}


\begin{prop}
\label{Residue at infinity}
Let $G$ be a meromorphic Drinfeld modular form of weight $2$, type $1$ for $\Gamma_0(T)$. 
\begin{enumerate}
\item  \label{residue P1}
 If $G$ is holomorphic at $P\in \Gamma_0(T)\char`\\ \Omega$, then $\mathrm{Res}_P\ G(z)dz=0$.
\item \label{residue P2}
If the $u$-series expansion of $G$ at $\infty$ is given by
      $G(z)=\sum_{i\geq -n} a_{G,\infty} (i(q-1)+1)u^{i(q-1)+1},$ then 
      $\mathrm{Res}_\infty G(z)dz=-\frac{a_{G,\infty} (1)}{\tilde{\pi}}$,
      where $\tilde{\pi}$ is the Carlitz period.
\item \label{residue P3}
If the $u_0$-series expansion of $G|_{2,1}\psmat{0}{-1}{1}{0}$ at $\infty$ is given by $\sum_{i\geq -n} a_{G,0}(i(q-1)+1)u_0^{i(q-1)+1}$ where $u_0:=u(z/T)$,
      then $\mathrm{Res}_0\ G(z)dz=-\frac{T}{\tilde{\pi}}a_{G,0}(1)$. 
      Moreover, if the order of vanishing of $G$ at $0$ is at least $2$, then $\mathrm{Res}_0\ G(z)dz=0$.
\end{enumerate}
\end{prop} 

\begin{proof}
\begin{enumerate}
\item By~\cite[\S 2.10]{GR96}, if $G$ is holomorphic at $P\in \Gamma_0(T)\char`\\ \Omega$, 
      then $G(z)dz$ is also holomorphic at $P$. Hence $\mathrm{Res}_P\ G(z)dz=0$.
\item The parameter at $\infty$ is given by  $u=\frac{1}{e_L(\tilde{\pi}z)}$ which implies  
      $dz=-\frac{1}{\tilde{\pi}} u^{-2}du$. Thus $G(z)dz = - \frac{1}{\tilde{\pi}} \cdot \sum_{i\geq -n} a_{G,\infty} (i(q-1)+1)u^{i(q-1)-1}du$. The coefficient of $u^{-1}$ gives the required result.
\item The parameter at $0$  is given by $u_0(z): =\frac{1}{e_L(\tilde{\pi}z/T)}$ which implies that
      $dz=-\frac{T}{\tilde{\pi}} u_0^{-2}du_0$. Thus
      $\big(G(z)|_{2,1}\psmat{0}{-1}{1}{0} \big)dz =- \frac{T}{\tilde{\pi}} \cdot \sum_{i\geq -n} a_{G,0}(i(q-1)+1)u_0^{i(q-1)-1}du_0$.
      Since the matrix $\psmat{0}{-1}{1}{0}$ permutes the cusps $0$ and $\infty$, 
      by comparing the coefficient of $u_0^{-1}$ on both sides,
      we get  $\mathrm{Res}_0\ G(z)dz=-\frac{T}{\tilde{\pi}}a_{G,0}(1)$.
      If the order of vanishing of $G$ at $0$ is $\geq 2$, then $G(z)|_{2,1}\psmat{0}{-1}{1}{0} =\sum_{i\geq 1} a_{G,0}(i(q-1)+1)u_0^{i(q-1)+1}$, which implies $\mathrm{Res}_0\ G(z)dz=0$.
      \end{enumerate}
\end{proof}
Now, we are in a position to state and prove the main result of this section. 
\begin{thm}
\label{MT_1}
Let $f\in M_{k,l}(\Gamma_0(T))$ be a non-zero Drinfeld modular form
such that the $u$-series expansion at $\infty$ belongs to $A[[u]]$. 
Suppose the $u$-series expansion of $fE_T^{q-l}$  at $\infty$ is given by
$fE_T^{q-l}= \sum_{j\geq 0} a_{fE_T^{q-l}}(j(q-1)+1)u^{j(q-1)+1}$ where $E_T$ is as in~\eqref{E_T}.
Fix $d \in \N$. Let  $b \in \N$ such that $r_{k, l}+2+a\  \frac{q^d-1}{q-1}=p^b$ holds
for some integer $a \geq 0$.
Then the following congruence  
$$a_{fE_T^{q-l}}(p^b(q-1)+1)\equiv 0 \pmod {T^{q^d}-T}$$
holds. 
Moreover, if $a$ can chosen to be $0$, then $a_{fE_T^{q-l}}(p^b(q-1)+1)=0$.
\end{thm}

\begin{proof}

For any integer $a \geq 0$, consider the function $G(a):=\frac{g_d^a}{\Delta_T^{a.\frac{q^d-1}{q-1}}}.\frac{h}{\Delta_T^{r_{k, l}+1}}.\frac{f}{E_T^l}$. 
\begin{itemize}
 \item  Since $g_d$ is non-vanishing at $\infty$ and $\Delta_T$ vanishes only at $\infty$,
        the function $\frac{g_d^a}{\Delta_T^{a.\frac{q^d-1}{q-1}}}$ is holomorphic on $\Omega \cup \{0\}$ and the possible pole is  only at $\infty$. 
 \item  The function $h$ vanishes exactly once (resp., $q$-times) at $\infty$ (resp., at $0$) as 
        a Drinfeld modular form for $\Gamma_0(T)$. The function $\Delta_T$  
        vanishes only at $\infty$ and the order of vanishing at $\infty$ is $q-1$ (cf. Proposition~\ref{IMP_Prop}(\ref{P1})). This implies that the function $\frac{h}{\Delta_T^{r_{k, l}+1}}$ is holomorphic on $\Omega \cup \{0\}$ with order of vanishing at $0$ is $q$ and has a pole only at  $\infty$.
 \item  By Proposition \ref{IMP_Prop}(\ref{P4}) we get $\frac{f}{E_T^l}\in M_{k-2l,0}(\Gamma_0(T))$ is a holomorphic on $\Omega\cup \{0,\infty\}$. 
\end{itemize}
Using the above properties, we get  the function $G(a)$ is a meromorphic Drinfeld modular form of weight $2$ and type $1$ for $\Gamma_0(T)$. Moreover, the function $G(a)$ is holomorphic on $\Omega\cup \{0\}$ with order of vanishing at least $q$ at $0$ and has a pole only at $\infty$. In particular, 
by Proposition \ref{Residue at infinity}, we obtain 
\begin{itemize}
 \item $\mathrm{Res}_P\ G(a)(z)dz=0$ \ for $P\in \Gamma_0(T)\char`\\ \Omega$, $\mathrm{Res}_0\ G(a)(z)dz=0$,
 \item $\mathrm{Res}_\infty\ G(a)(z)dz=-\frac{a_{G(a)}(1)}{\tilde{\pi}}$,
  where $a_{G(a)}(1)$ is the coefficient of $u$ in the $u$-series expansion of $G$ at $\infty$.
\end{itemize}
By Theorem~\ref{Residue Theorem}, we get  $\mathrm{Res}_\infty G(a)(z)dz= 0$
and hence $a_{G(a)}(1)=0$. 
By Proposition~\ref{IMP_Prop}((\ref{P5})\&(\ref{P7})) we have $h= -\frac{E_T^q}{\Delta_T}$. Consequently, 
$G(a) = -\frac{g_d^aE_T^{q-l}f}{\Delta_T^{r_{k, l}+a.\frac{q^d-1}{q-1}+2}}.$
Since $g_d\equiv 1 \pmod {T^{q^d}-T}$, we get
 \begin{equation}
 \label{congruence mod product of degree d}
-G(a) = \frac{g_d^aE_T^{q-l}f}{\Delta_T^{r_{k, l}+a.\frac{q^d-1}{q-1}+2}} \equiv 
      \frac{E_T^{q-l}f}{\Delta_T^{r_{k, l}+a.\frac{q^d-1}{q-1}+2}}\pmod {T^{q^d}-T}.
\end{equation}
Now, we  begin the proof of the theorem.   By hypothesis, let $b\in \N$  such that 
$r_{k, l}+a.\frac{q^d-1}{q-1}+2 =p^b $
holds for some $a \geq 0$. With this choice of $a$, ~\eqref{congruence mod product of degree d} becomes
\begin{equation}
 \label{derived congruence mod product of degree d}
-G(a) =\frac{g_d^aE_T^{q-l}f}{\Delta_T^{p^b}} \equiv 
      \frac{E_T^{q-l}f}{\Delta_T^{p^b}}\pmod {T^{q^d}-T}.
\end{equation}
Since $\frac{1}{\Delta_T^{p^b}}= u^{-p^b(q-1)}(1-u^{p^b(q-1)^2}+\ldots)$, an easy computation shows that the coefficient of $u$ in the $u$-series expansion of $\frac{E_T^{q-l}f}{\Delta_T^{p^b}}$ 
at  $\infty$ is $a_{fE_T^{q-l}}(p^b(q-1)+1)$.  By~\eqref{derived congruence mod product of degree d}, we obtain
$a_{fE_T^{q-l}}(p^b(q-1)+1) \equiv - \pmod {T^{q^d}-T}.$
Since $a_{G(a)}(1)=0$, we are done.
\end{proof}
\begin{remark}
When $d=1$, the existence of $a\geq 0$ with $r_{k,l}+2+a=p^b$ is automatic if $b \gg 0$.
Theorem~\ref{MT_1} is a generalization of~\cite[Theorem 3.4]{Cho09} 
from trivial type to arbitrary type.
\end{remark}

Now, we have a corollary of Theorem~\ref{MT_1}.
\begin{cor}
\label{Choi_Corollary}
Let $f\in M_{k,l}(\Gamma_0(T))$ be a Drinfeld modular form such that the $u$-series expansion at $\infty$ is given by
$f=\sum_{j\geq 0} a_f(j(q-1)+l)u^{j(q-1)+l}\in A[[u]]$. If $p^\alpha \mid l$ for some $\alpha \in \N$ and
 $m \leq \alpha$ be a natural number such that $p^m>r_{k, l}+1$, then the congruence
$$a_f((p^m-1)(q-1)+l)\equiv 0 \pmod {T^q-T} \ holds.$$
\end{cor}

\begin{remark}
Corollary~\cite[Corollary 3.5]{Cho09} is a special case of Corollary~\ref{Choi_Corollary} for $l=0$ and $m=1$. In fact,  the assumption $q > p+1$ in Corollary $\it{loc. cit.}$ is redundant.
\end{remark}
\begin{proof}[Proof of Corollary~\ref{Choi_Corollary}]
Let $q-l=p^\alpha x$ for some $x\in \N$. Now
\begin{align*}
&\sum_{n\geq 0} a_{fE_T^{p^\alpha x}}( n(q-1)+1)u^{n(q-1)+1}=fE_T^{p^\alpha x}\\
&= \sum_{j\geq 0} a_f(j(q-1)+l)u^{j(q-1)+l}. \Big(\sum_{i\geq 0} (a_{E_T}(i(q-1)+1))^{p^\alpha}u^{p^\alpha i(q-1)+p^\alpha}\Big)^x.
\end{align*}
By comparing the coefficients of $u^{p^m(q-1)+1}$ on both sides, we get   
$$a_{fE_T^{q-l}}(p^m(q-1)+1)= a_f((p^m-1)(q-1)+l)a_{E_T}(1)^{p^\alpha x}=a_f((p^m-1)(q-1)+l),$$
since $a_{E_T}(1)=1$. The first equality follows from the fact that $p^\alpha (q-1)+p^\alpha>p^m(q-1)+1$.
Now, the result follows from Theorem~\ref{MT_1} by taking $d=1,b=m$.  
\end{proof}

Now, we shall give an example satisfying Corollary~\ref{Choi_Corollary}.  
\begin{exple}
Let $p=3, q=3^2, l=6$ and $k=12$, the assumptions of Corollary \ref{Choi_Corollary} are satisfied and the space $M_{12,6}(\Gamma_0(T))$ is generated by $E_T^6$. 
Hence, by Corollary~\ref{Choi_Corollary}, we get that 
$a_{E_T^6}((p-1)(q-1)+6) \equiv 0\ (T^q-T).$ In fact, a straight forward calculation shows that  $a_{E_T^{pm}}((p-1)(q-1)+pm)=0$ for any $m\in \N$ with $pm<q-1$. 
\end{exple}
We can also produce infinitely many examples satisfying Corollary~\ref{Choi_Corollary} as follows:
\begin{exple}
Suppose $p|l$. Consider the function $g_1^{p-2}E_T^l\in M_{(q-1)(p-2)+2l,l}(\Gamma_0(T))$.
Since $p> r_{(q-1)(p-2)+2l, l}+1$,  we get $a_{g_1^{p-2}E_T^l}((p-1)(q-1)+l)\equiv 0 \pmod {T^q-T}$
by Corollary~\ref{Choi_Corollary}.
\end{exple}

\section{linear relations between the initial Fourier coefficients}
\label{Sec_Linear Relations}
In this section, we determine all the  linear relations between the initial coefficients in $u$-series expansion of Drinfeld modular forms of any weight,  arbitrary type for  $\Gamma_0(T)$.
For any integer $i \geq 0$, consider the map
$$a_i^*: M_{k,l}(\Gamma_0(T)) \rightarrow C \  \ \mathrm{defined \ by} \ 
f \rightarrow a_f(i(q-1)+l),$$
where $f=\sum_{0\leq j\equiv l \pmod {q-1}}^\infty a_f(j)u^{j}.$
Then $a_i^* \in (M_{k,l}(\Gamma_0(T)))^*$, the dual of $M_{k,l}(\Gamma_0(T))$. The elements of $S$ in Proposition~\ref{IMP_Prop}(\ref{P3}) satisfy $a_i^*(\Delta_W^{r_{k,l}-j}\Delta_T^j E_T^l)=0$ for $i<j$, which implies 
\begin{prop}
\label{Basis for the dual}
The set $\{a_i^*\}_{i=0}^{r_{k,l}}$ forms a basis for the $C$-vector space
$(M_{k,l}(\Gamma_0(T)))^*$.
\end{prop}

For any integer $N\geq 0$, consider the surjective map  
$$\psi_{k,l;N} : C^{r_{k, l,N}+1} \longrightarrow (M_{k,l}(\Gamma_0(T)))^* \ \ \mathrm{defined \ by} \ 
(c_0,c_1,\ldots,c_{r_{k, l,N}}) \longrightarrow \sum_{i=0}^{r_{k, l,N}}c_i a_i^*.$$
\begin{dfn}
\label{definition of L_k,l;N(Gamma_0T)}
For $k \in \N$, $l\in \Z/(q-1)\Z$ such that $k\equiv 2l \pmod {q-1}$, we define
$L_{k,l;N}(\Gamma_0(T)):= \mrm{Ker}(\psi_{k,l;N})$. By Proposition~\ref{Basis for the dual}, $\dim_C L_{k,l;N}(\Gamma_0(T))= N+1$.
\end{dfn}

Now, we state the main result of this section, which is a generalization of~\cite[Theorem 4.1]{Cho09} from trivial type to arbitrary type.  For each Drinfeld modular form $g\in M_{N(q-1)+2l,l}(\Gamma_0(T))$, we define the elements $b(k,l,N,g;i)$, $c(k,l,N,g;i)\in C$ by the $u$-series expansion
\begin{equation}
\label{required function for making weight 2 type 1 form}
\frac{hg}{\Delta_T^{r_{k, l,N}}E_T^{2l}}=\sum_{i=0}^{r_{k, l,N}}b(k,l,N,g;i)u^{-i(q-1)+1-l}+ \sum_{i=1}^\infty c(k,l,N,g;i)u^{i(q-1)+1-l}.
\end{equation}

\begin{thm}
\label{MT_2}
The map $\phi_{k,l;N}: M_{N(q-1)+2l,l}(\Gamma_0(T)) \longrightarrow L_{k,l;N}(\Gamma_0(T))$ defined by
$\phi_{k,l;N}(g)  = (b(k,l,N,g;0),b(k,l,N,g;1),\ldots, b(k,l,N,g;r_{k, l,N}) )$
is an isomorphism of $C$-vector spaces. 
\end{thm}


\begin{proof}
First we  show that the image of $\phi_{k,l;N}$ belongs to $L_{k,l;N}(\Gamma_0(T))$. 
Let $g\in M_{N(q-1)+2l,l}(\Gamma_0(T))$.  Then, for any $f=\sum_{j=0}^\infty a_f(j(q-1)+l)u^{j(q-1)+l} \in M_{k,l}(\Gamma_0(T))$, consider the function
$G:= \frac{hg}{\Delta_T^{r_{k, l,N}}E_T^{2l}} \cdot f= 
\frac{h}{\Delta_T^{r_{k, l,N}}}\cdot \frac{g}{E_T^l} \cdot \frac{f}{E_T^l}$. 

\begin{itemize}
 \item The function $h$ vanishes exactly once (resp., $q$-times) at $\infty$ (resp., at $0$) and $\Delta_T$ vanishes only at $\infty$ and the order of vanishing is $q-1$
       (Proposition~\ref{IMP_Prop}(\ref{P1})). This implies that $\frac{h}{\Delta_T^{r_{k, l,N}}}$ is holomorphic on $\Omega \cup \{0\}$ with order of vanishing $q$ at $0$ and a pole only at $\infty$.
 \item By Proposition~\ref{IMP_Prop}(\ref{P4}), $\frac{g}{E_T^l}\in M_{N(q-1),0}(\Gamma_0(T))$ and $\frac{f}{E_T^l}\in M_{k-2l,0}(\Gamma_0(T))$.
\end{itemize}
These properties imply that the function $G$ is a meromorphic Drinfeld modular form of weight $2$ and type $1$ for $\Gamma_0(T)$. Moreover, the function $G$ is holomorphic on $\Omega\cup\{0\}$ with order of vanishing at least $q$ at $0$ and has a pole only at $\infty$. Note that, the coefficient of $u$ in the $u$-series expansion of $G$ is $\sum_{i=0}^{r_{k, l,N}}b(k,l,N,g;i)a_f(i(q-1)+l)$, which is equal to $0$
by Theorem~\ref{Residue Theorem} and Proposition~\ref{Residue at infinity}.
Thus the image of $\phi_{k,l;N}$ belongs to $L_{k,l;N}(\Gamma_0(T))$. 
Clearly, $\phi_{k,l;N}$  is linear.

Now, we show that $\phi_{k,l;N}$ is injective. Suppose that $\phi_{k,l;N}(g)=0$ for some $g\in M_{N(q-1)+2l,l}(\Gamma_0(T))$. For any $f\in M_{k,l}(\Gamma_0(T))$, by~\eqref{required function for making weight 2 type 1 form}, we get
\begin{equation}
\frac{hg}{\Delta_T^{r_{k, l,N}}E_T^{2l}}.f= \sum_{i=1}^\infty c(k,l,N,g;i)u^{i(q-1)+1-l}\cdot\sum_{j=0}^\infty a_f(j(q-1)+l)u^{j(q-1)+l}.
\end{equation}
Therefore $\frac{hgf}{\Delta_T^{r_{k, l,N}}E_T^{2l}}$ is a doubly cuspidal holomorphic Drinfeld modular form of weight $2$, type $1$ for $\Gamma_0(T)$. Since the modular curve $X_0(T)$ has genus $0$, the function 
$\frac{hgf}{\Delta_T^{r_{k, l,N}}E_T^{2l}}$ is identically zero.  This forces $g = 0$. 
Hence the map $\phi_{k,l;N}$ is injective. Since the dimensions of 
$L_{k,l;N}(\Gamma_0(T))$, $M_{N(q-1)+2l,l}(\Gamma_0(T))$ are equal,
the map $\phi_{k,l;N}$ is an isomorphism.
\end{proof}
The Theorem above and  Proposition~\ref{IMP_Prop}((\ref{P3}) and (\ref{P4})), we get
\begin{cor}
The $C$-vector spaces $L_{k,l;N}(\Gamma_0(T))$ and $L_{k-2l,0;N}(\Gamma_0(T))$ are isomorphic.
\end{cor}
We conclude the article with a remark that one can  extend \cite[Theorem 4.1]{Cho08} for $\GL_2(A)$ from trivial type to arbitrary type by replacing the function $\frac{g_1^{q-\alpha}}{h^{(q-1)(r+N-1)}}u$ in \cite[Page 99]{Cho08} by $\frac{g_1^{q-\beta}}{h^{(q-1)(\dim_C  M_{k,l}(\GL_2(A))+N)+2l-1}}f$,  where $f\in M_{N(q^2-1)+l(q+1),l}(\GL_2(A))$ and $\beta = \frac{k-(q+1)l}{q-1}+(1-\dim_C M_{k,l}(\GL_2(A)))(q+1)$.

\bibliographystyle{plain, abbrv}

\end{document}